\documentclass[11pt]{amsart}

\usepackage[letterpaper, margin=1in]{geometry} %% Letter setup

\usepackage{etex} %% This is for an extension aind allows us to use many packages (if you know what I mean).
\usepackage{amsmath} % Necessary
\usepackage{amssymb} % Necessary
\usepackage{amsthm} % Necessary
\usepackage{amsfonts} % Necessary
\usepackage{booktabs} % For tables that look nice with the command \begin{tabular}
\usepackage[usenames,dvipsnames,svgnames]{xcolor} % For tonnes of colours and control over creating colours
\usepackage{hyperref} % Automatically turns all internal references to links
\hypersetup{colorlinks=true,
allcolors=blue!75!black} % Changes links to be dark blue with no annoying boxes around them
\usepackage{mathrsfs} % Allows the use of \mathscr letters
\usepackage{enumerate}
\usepackage{pinlabel}
\usepackage{tikz}
\usepackage{tikz-cd}
\usepackage{caption}
\usetikzlibrary{decorations.markings}

\newcommand{\RR}{\mathbb{R}}
\newcommand{\ZZ}{\mathbb{Z}}

\newcommand{\vC}{\mathcal{C}}

\newcommand{\vP}{\mathcal{P}}

\newcommand{\Aut}{\operatorname{Aut}}

\newcommand{\Homeo}{\operatorname{Homeo}}

\newcommand{\Mod}{\operatorname{Mod}}

\newcommand{\lra}{\longrightarrow}

\newcommand{\ol}[1]{\overline{#1}}

\newcommand{\wt}[1]{\widetilde{#1}}

\definecolor{lightgrey}{gray}{.85}

\theoremstyle{definition}

\newtheorem*{rmk}{Remark}

\theoremstyle{plain}
\newtheorem{thm}{Theorem}
\newtheorem{lem}[thm]{Lemma}
\newtheorem{cor}[thm]{Corollary}
\newtheorem{prop}[thm]{Proposition}

\newtheorem{conj}[thm]{Conjecture}

\title{Cofinal elements and fractional Dehn twist coefficients}
\author[Adam Clay]{Adam Clay}
\thanks{Adam Clay was partially supported by NSERC grant RGPIN-05343-2020.}
\address{Department of Mathematics\\
University of Manitoba \\
Winnipeg \\
MB Canada R3T 2N2} \email{Adam.Clay@umanitoba.ca}
\urladdr{http://server.math.umanitoba.ca/~claya/} 

\author[Tyrone Ghaswala]{Tyrone Ghaswala}
\thanks{Tyrone Ghaswala was supported by a CIRGET postdoctoral fellowship at L'Universit\'{e} du Qu\'{e}bec \`{a} Montr\'{e}al}
\address{Centre for Education in Mathematics and Computing\\
University of Waterloo \\
Waterloo \\
ON Canada N2L 3G1} \email{tghaswala@uwaterloo.ca}
\urladdr{https://tyjgh.github.io/} 
\date{\today}

\begin{document}

\maketitle

%%%%%%%%%%%%%%%%%%%%%%%%%%%%%%%%%%%
\begin{abstract}
We show that for a surface $S$ with positive genus and one boundary component, the mapping class of a Dehn twist along a curve parallel to the boundary is cofinal in every left ordering of the mapping class group $\Mod(S)$.  We apply this result to show that one of the usual definitions of the fractional Dehn twist coefficient---via translation numbers of a particular action of $\Mod(S)$ on $\mathbb{R}$---is in fact independent of the underlying action when $S$ has genus larger than one.  As an algebraic counterpart to this, we provide a formula that recovers the fractional Dehn twist coefficient of a homeomorphism of $S$ from an arbitrary left ordering of $\Mod(S)$.
\end{abstract}

\maketitle

\section{Introduction}
 
Braid groups, and more generally mapping class groups of hyperbolic surfaces with nonempty boundary, are left-orderable.  There are many techniques for producing explicit examples of such orderings, ranging from combinatorial conditions on representative words relative to a certain generating set, to conditions on arc diagrams in the surface, to hyperbolic geometry (see e.g. \cite{DDRW08, RW00}).  In fact there are uncountably many ways to left order any left-orderable mapping class group, aside from $B_2 \cong \mathbb{Z}$ which only admits two left orderings. Despite this flexibility in creating left orderings of mapping class groups, the left orderings all display a type of algebraic rigidity that is also reflected in the dynamics of their actions on $\mathbb{R}$.  

Given a left-ordered group $(G,<)$, an element $g \in G$ is called \emph{cofinal} relative to the ordering $<$ of $G$ if 
\[G = \{ h \in G \mid \exists k \in \mathbb{Z} \mbox{ such that } g^{-k} < h < g^k \}.\]
In terms of dynamics, this means that for every action of $G$ on $\mathbb{R}$ by orientation-preserving homeomorphisms and without global fixed points, the element $g$ will act without fixed points.  

Using $\Mod(\Sigma_g^1)$ to denote the mapping class group of a surface of genus $g$ having one boundary component and $T_d$ to denote the Dehn twist about a curve isotopic to the boundary of $\Sigma_g^1$, we prove:

\begin{thm}
\label{thm:cofinal-boundary-twist}
For all $g>0$, the element $T_d \in \Mod(\Sigma_g^1)$ is cofinal in every left ordering of $\Mod(\Sigma_g^1)$.
\end{thm} 

For the case of $g=0$, the braid groups, the above result also holds and is well-known \cite[Proposition 3.6]{DDRW08}.  Aside from the dynamical consequences explored in this paper, it is also worth noting that Theorem \ref{thm:cofinal-boundary-twist} implies that every positive cone of $\Mod(\Sigma_g^1)$ is a coarsely connected subset of the Cayley graph of $\Mod(\Sigma_g^1)$ \cite[Lemma 4.14]{Hucho}.  In the language of \cite{Hucho}, this means that $\Mod(\Sigma_g^1)$ for $g>0$ is an example of a \emph{Prieto group}.

There is a classical correspondence between elements of $H^2(G;\ZZ)$ and equivalence classes of central extensions of $G$ by $\ZZ$ (see \cite[Chapter 4]{brown}).  In the setting of ordered groups, this plays out as a correspondence between circularly-ordered groups (or groups admitting an action on $S^1$ by orientation-preserving homeomorphisms) and left-ordered central extensions with cofinal central elements (or central extensions admitting an action on $\RR$ by orientation-preserving homeomorphisms, for which the central element acts as translation by one).

We can apply this to mapping class groups via the well-known central extension given by the ``capping homomorphism," using Theorem \ref{thm:cofinal-boundary-twist} to establish a correspondence between left orderings of $\Mod(\Sigma_g^1)$ and circular orderings of $\Mod(\Sigma_{g,1})$, the mapping class group of a surface with one marked point.  By combining this correspondence with a rigidity result of Mann and Wolff \cite{MW}, we are able to conclude that all left orderings of $\Mod(\Sigma_g^1)$ have certain dynamical properties in common, and apply this to fractional Dehn twist coefficients as follows.

We denote the fractional Dehn twist coefficient of $h \in \Mod(\Sigma_g^1)$ by $c(h)$, and first remark that although this quantity is often defined in terms of singular foliations, it can equivalently be defined as the translation number of $h$ under a particular action of $\Mod(\Sigma_g^1)$ on $\mathbb{R}$ (E.g. see \cite{IK17} or \cite{Malyutin04}).  By an application of Theorem \ref{thm:cofinal-boundary-twist} and the result of Mann and Wolff, we show that this definition is independent of the choice of action.

\begin{thm}
\label{conj to shift intro version} 
Suppose that $g \geq 2$ and let $\rho:\Mod(\Sigma_g^1) \rightarrow \mathrm{Homeo}_+(\mathbb{R})$ be an injective homomorphism such that the action of $\Mod(\Sigma_g^1)$ on $\mathbb{R}$ is without global fixed points. Then, up to reversing orientation, $\rho$ is conjugate to a representation $\rho':\Mod(\Sigma_g^1) \rightarrow \mathrm{H}\widetilde{\mathrm{ome}}\mathrm{o}_+(S^1)$ such that $\rho'(T_d)(x) = x+1$ for all $x \in \mathbb{R}$, and for every $h \in \Mod(\Sigma_g^1)$ the fractional Dehn twist coefficient $c(h)$ is given by the translation number of $h$ as computed from the $\rho'$-action on $\mathbb{R}$.
\end{thm}

Morally, Theorem \ref{conj to shift intro version} tells us that the translation number of an element of $\Mod(\Sigma_g^1)$ is intrinsic to the element itself, and does not depend on any particular choice of action.  We also see that since the fractional Dehn twist coefficient is a rational number \cite[Section 3]{HKM}, it follows that for every action of $\Mod(\Sigma_g^1)$ on $\RR$ with $T_d$ acting as translation by one, all translation numbers are rational. 

One weakness of Theorem \ref{conj to shift intro version}, however, is that in order to compute the fractional Dehn twist coefficient from a representation $\rho:\Mod(\Sigma_g^1) \rightarrow \mathrm{Homeo}_+(\mathbb{R})$, one must first normalise so that the Dehn twist $T_d$ is sent to translation by one. This normalisation issue disappears when we re-cast the previous theorem in terms of left orderings.  

Given a left ordering $<$ of $\Mod(\Sigma_g^1)$ for which $T_d > id$ and $h \in \Mod(\Sigma_g^1)$, we use $[h]_<$ to denote the unique power of $T_d \in \Mod(\Sigma_g^1)$ such that $T_d^{[h]_<} \leq h < T_d^{[h]_<+1}$.  Such a power exists by Theorem \ref{thm:cofinal-boundary-twist}.

\begin{thm}
\label{fdtc from orders}
Suppose that $g\geq 2$ and $h \in \Mod(\Sigma_g^1)$. Denote the fractional Dehn twist coefficient of $h$ by $c(h)$.  For every left ordering $<$ of $\Mod(\Sigma_g^1)$ for which $T_d > id$, we have
\[ c(h) = \lim_{n \to \infty} \frac{[h^n]_<}{n}.
\] 
\end{thm}

This theorem allows one to easily find bounds on fractional Dehn twist coefficients using left-orderings (Proposition \ref{estimation}), and in some cases allows one to compute the precise value of the fractional Dehn twist coefficient of a particular homeomorphism from only two inequalities (Corollary \ref{estimation cor}).

\subsection{Outline of the paper}
In Section \ref{orders and dynamics} we review left orderings of groups, their relationship to circular orderings, as well as translation number, rotation number and semiconjugacy.  In Section \ref{sec:mcg-basics} we introduce mapping class groups, and prepare several lemmas. Theorem \ref{thm:cofinal-boundary-twist} is proved in Section \ref{sec:cofinal}. In Section \ref{FDTC and orders} we define the fractional Dehn twist coefficient, prove Theorems \ref{conj to shift intro version} and Theorem \ref{fdtc from orders}, and provide some basic tools for estimating fractional Dehn twist coefficients from these results.  In Section \ref{limitations} we provide examples that show the limitations of these theorems.

\subsection{Acknowledgements} We would like to thank Alan McLeay for his helpful insights into Ivanov's theorem and the automorphism group of the mapping class group. We also thank the referees for their careful reading and suggestions for improvement.

\section{Orders, dynamics and semiconjugacy}
\label{orders and dynamics}
\subsection{Left-orderable and circularly-orderable groups}
\label{lifts and quotients}

A \emph{left ordering} of a group $G$ is a strict total ordering $<$ of the elements of $G$ such that $g<h$ implies $fg<fh$ for all $f, g, h \in G$.  A \emph{positive cone} $P \subset G$ is a semigroup satisfying $G\setminus\{ id \} = P \cup P^{-1}$.  We can pass from an ordering $<$ to a positive cone $P_<$ by setting $P_< = \{ g \in G \mid g> id\}$, and from a positive cone $P$ to an ordering $<_P$ by declaring $g<h \iff g^{-1}h \in P$.  It is straightforward to check that this correspondence is a bijection.  When $G$ admits a left ordering it will be called a \emph{left-orderable} group. When $G$ comes equipped with a prescribed left ordering, we will refer to it as a \emph{left-ordered group} and we denote such objects as a pair $(G, <)$.

Given a left-ordered group $(G,<)$, a \emph{$<$-cofinal} set $S \subset G$ (or simply ``cofinal" if the ordering of $G$ is understood) is a subset of $G$ satisfying
\[ G = \{ g \in G \mid \exists s, t \in S \mbox{ such that } s < g < t \}.
\]
An element $g \in G$ is called $<$-cofinal if the cyclic subgroup $\langle g \rangle$ is $<$-cofinal.

A \emph{circular ordering} of a group $G$ is a function $f:G^2 \rightarrow \{0, 1\}$ satisfying:
\begin{enumerate}[(i)]
\item $f(g, g^{-1}) = 1$ for all $g \in G \setminus \{ id \}$;
\item $f(id, g) = f(g, id) = 0$ for all $g \in G$;
\item $f$ is an inhomogeneous cocycle, that is $$f(g_2, g_3) - f(g_1g_2, g_3) + f(g_1, g_2g_3) - f(g_1, g_2) = 0$$ for all $g_1, g_2, g_3 \in G$.
\end{enumerate}

With this definition of circular ordering in hand, we define \emph{circularly-orderable} groups and \emph{circularly-ordered} groups in the obvious way.

\begin{rmk}
A circular ordering on a group is more commonly defined as a homogeneous cocycle $c:G^3 \to \{0, \pm 1\}$ satisfying $c(g_1,g_2,g_3) = 0$ if and only if $g_i = g_j$ for some $i \neq j$. However, for ease of exposition in cohomological arguments it is often more straightforward to define circular orderings in terms of inhomogeneous cocycles as we have done above. That these definitions are equivalent can be found in \cite[Proposition 2.3]{CG21}. 
\end{rmk}

We now construct a left-ordered central extension $(\widetilde{G}_f, <_f)$ of any circularly-ordered group $(G, f)$.  Let $\widetilde{G}_f = G \times \mathbb{Z}$ as a set, and equip it with the operation
\[ (g, n)(h, m) = (gh, n+m+f(g,h)).
\]
This group comes equipped with a left ordering $<_f$ whose positive cone is $\{(g,n) \mid n \geq 0\}$, and a canonical positive, central cofinal element $z_f = (id, 1)$.   To our knowledge, this construction first appears in \cite{Zeleva}. However, the underlying group construction is nothing more than the usual correspondence between elements of $H^2(G; \mathbb{Z})$ and equivalence classes of central extensions 
\[ 0 \rightarrow \mathbb{Z} \rightarrow \widetilde{G} \rightarrow G \rightarrow \{ id\}.
\]
For details of this correspondence, see \cite[Chapter 4]{brown}.

It is similarly possible to begin with a left-ordered group $G$ admitting a left ordering $<$ and a positive, cofinal, central element $z \in G$, and to construct a circular ordering $f_<$ on $G/\langle z \rangle$ according to the following rule.  Given $g \in G$, define $\{g \}$ to be the unique coset representative of $g \langle z \rangle$ satisfying $id \leq  \{ g\} < z$, and define $f_<$ according to $\{g\}\{h\} = \{gh\} z^{f_<(g\langle z \rangle , h \langle z \rangle)}$.   That this defines a circular ordering of $G/\langle z \rangle$ can be checked from the definition.

These two constructions are not inverse to one another, though applying the lift and quotient operations successively does yield a group that is naturally isomorphic to the original for categorical reasons (similarly when one applies the quotient and lift operations successively).  See \cite[Proposition 2.9]{CG} for details.
%, and Lemma \ref{canonical isom} below.

\subsection{Dynamic realisations and tight embeddings}

For a countable, left-ordered group $(G, <)$, we recall the notion of dynamic realisation as in \cite{BC}, see also \cite{DNR}. 

A \textit{gap} in $(G, <)$ is a pair of elements $g, h \in G$ with $g<h$ such that no element $f \in G$ satisfies $g<f<h$.  An order-preserving embedding $t: G \rightarrow \mathbb{R}$ is called a  \textit{tight embedding} of $(G, <)$ if whenever $(a, b) \subset \mathbb{R} \setminus t(G),$ there exists a gap $g, h \in G$ such that $(a,b) \subset (t(g), t(h))$.  We further require, for ease of exposition, that our tight embeddings satisfy $t(id) =0$.  One can check that the usual construction of $t:G \rightarrow \mathbb{R}$ in the definition of the dynamic realisation, using countability of $G$ to inductively construct an embedding by choosing midpoints of previously defined intervals as in \cite[Section 2.4]{CR16} or \cite[Chapter 1]{DNR}, provides an example of a tight embedding.

Define a \emph{dynamic realisation} $\rho_{<} : G \rightarrow \mathrm{Homeo}_+(\mathbb{R})$ of $(G, <)$ by setting $\rho_{<}(g)(t(h)) = t(gh)$ for all points in the image of $t$, using continuity to extend to points in $ \overline{t(G)}$, and then extending affinely to $\mathbb{R} \setminus \overline{t(G)}$.  This construction is also independent of choice of tight embedding, in the sense that different choices of tight embedding will yield conjugate dynamic realisations \cite[Proposition 3.1]{BC}.  The essential property of dynamic realisations is that they allow one to recover the ordering $<$ of $G$ by examining the orbit of $t(id) = 0$:
\[(\forall g, h \in G) [g<h \iff t(g) < t(h) \iff \rho_{<}(g)(0) < \rho_{<}(h)(0)].
\]

We now extend this to the realm of circular orderings.  Given a circular ordering $f$ of a countable group $G$, define the \emph{dynamic realisation} $\rho_f:G \rightarrow \mathrm{Homeo}_+(S^1)$ as follows.  

Choose a tight embedding $\tilde{t}: \widetilde{G}_f \rightarrow \mathbb{R}$, with associated dynamic realisation $\widetilde{\rho}_f : \widetilde{G}_f \rightarrow \mathrm{Homeo}_+(\mathbb{R})$.  Note that since $\langle z_f \rangle$ is unbounded in $\widetilde{G}_f$, the image $\tilde{t}(\langle z_f \rangle)$ is similarly unbounded in $\mathbb{R}$.  As such, the map $\widetilde{\rho}_f(z_f) : \mathbb{R} \rightarrow \mathbb{R}$ acts without fixed points, since $\widetilde{\rho}_f(z_f)(\tilde{t}(z_f^k)) = \tilde{t}(z_f^{k+1})$ for all $k \in \mathbb{Z}$.   Consequently this map is conjugate to one of  $sh(\pm 1) : \mathbb{R} \rightarrow \mathbb{R}$, where $sh(k)(x) = x+k$ for all $x \in \mathbb{R}$ and $k \in \mathbb{Z}$.  Since $z_f$ is positive in the left ordering of $\widetilde{G}_f$, it follows that $\widetilde{\rho}_f(z_f)$ is conjugate to $sh(1)$.  As such, we may assume (by applying the appropriate conjugation) that the tight embedding $\tilde{t}: \widetilde{G}_f \rightarrow \mathbb{R}$ satisfies $\tilde{t}(z_f^k) = k$  and $\tilde{t}(1) = 0$, and consequently that $\widetilde{\rho}_f(z_f) = sh(1)$.  Therefore we may assume that $\widetilde{\rho}_f : \widetilde{G}_f \rightarrow \mathrm{H}\widetilde{\mathrm{ome}}\mathrm{o}_+(S^1)$, where 
\[\mathrm{H}\widetilde{\mathrm{ome}}\mathrm{o}_+(S^1) = \{ f \in \mathrm{Homeo}_+(\mathbb{R}) \mid f(x+1) = f(x) +1 \}.
\]
%  \tg{Need to define  $\mathrm{H}\widetilde{\mathrm{ome}}\mathrm{o}_+(S^1)$}

Let $q : \mathrm{H}\widetilde{\mathrm{ome}}\mathrm{o}_+(S^1) \rightarrow \mathrm{Homeo}_+(S^1)$ denote the quotient map whose kernel consists of integral translations, and for arbitrary $g \in G$ let $\widetilde{g} = (g,0) \in \widetilde{G}_f$.   Define  $\rho_f:G \rightarrow \mathrm{Homeo}_+(S^1)$ by $\rho_f(g)(x) =  q(\widetilde{\rho}_f(\widetilde{g}))(x)$.  Note that 
\[ \rho_f(g) \circ \rho_f(h)(x) = q(\widetilde{\rho}_f(\widetilde{g} \widetilde{h}))(x),
\]
and that $\widetilde{g} \widetilde{h} = (g,0)(h,0) = (gh, f(g,h)) = (gh, 0)(1, f(g,h))$.  As such, 
\[ q(\widetilde{\rho}_f(\widetilde{g} \widetilde{h}))(x) = q(\widetilde{\rho}_f(\widetilde{gh}) \circ sh(f(g,h)))(x) = q(\widetilde{\rho}_f(\widetilde{gh}))(x),
\]
meaning $\rho_f:G \rightarrow \mathrm{Homeo}_+(S^1)$ is a homomorphism.   As before, this construction is defined up to conjugation.
%by a homeomorphism.

%\begin{prop}
%Suppose that $(G, f)$ is countable circularly-ordered group, and that $t, t' : \widetilde{G}_f \rightarrow \mathbb{R}$ are tight embeddings satisfying $t(z_f^k) = t'(z_f^k) =k$, and that these embeddings are used to construct dynamic realisations $\rho_{f}, \rho'_{f} : G \rightarrow \widetilde{\mathrm{Homeo}}_+(S^1)$ as above.  Then there exists an order-preserving homeomorphism $\phi : S^1 \rightarrow S^1$ such that 
%\[ \rho_{f}(g) = \phi \circ \rho'_{f}(g) \circ \phi^{-1}
%\]
%for all $g \in G$.
%\end{prop}
%\begin{proof}
%It follows from Proposition \ref{prop:conjugate} that $\widetilde{\rho}_f$ and $\widetilde{\rho}'_f$ are conjugate via some order-preserving homeomorphism $\psi : \mathbb{R} \rightarrow \mathbb{R}$.  Note further that  $\psi$ satisfies 
%\[ \widetilde{\rho}_f(z_f) \circ \psi = \psi \circ  \widetilde{\rho}'_f (z_f). 
%\]
%Here $\widetilde{\rho}_f(z_f) = \widetilde{\rho}'_f (z_f) = sh(1)$, so it follows that $\psi$ descends to the required order-preserving homeomorphism $\phi : S^1 \rightarrow S^1$.
%\end{proof}

Moreover, the map $t: = q \circ \tilde{t}$ provides an embedding $t: G \rightarrow S^1$; having fixed $\tilde{t}(id) = 0$ ensures that $\ker(q) \cap \tilde{t}(\widetilde{G}_f) = \langle z_f \rangle$.  This allows us to make a similar observation as in the case of dynamic realisations of left orderings:  Identify $S^1$ with $[0,1) \cong \mathbb{R}/ \mathbb{Z}$, and let $f_{S^1} : (S^1)^2 \rightarrow \{ 0, 1\}$ denote the standard circular ordering of $S^1$.    Then we have
\[ (\forall g, h \in G)[f(g,h) = f_{S^1}(t(g), t(h)) = f_{S^1}(\rho_f(g)(0), \rho_f(h)(0))]. 
\]
%Identifying $S^1$ with $\mathbb{R}/\mathbb{Z}$, we can define the dynamic realisation $\rho_f:G \rightarrow \mathrm{Homeo}_+(S^1)$ by $\rho_f(g)(x) = [\widetilde{\rho}_f(\widetilde{g})(\widetilde{x})]$, where:
%\begin{itemize}
%\item $\widetilde{g}$ is the element $(g,0) \in \widetilde{G}_f$, 
%\item $\widetilde{x}$ is an arbitrary lift of $x$, and 
%\item $[ \cdot ]$ denotes the equalence class in $\mathbb{R} / \mathbb{Z}$.
%\end{itemize}
%\textbf{Then we need to check that this all works.}

\subsection{Semi-conjugacy and bounded cohomology}

Representations $\rho:G \rightarrow \mathrm{Homeo}_+(S^1)$ are classified up to semiconjugacy according to their Euler class $eu(\rho) \in H^2_b(G ; \mathbb{Z})$ \cite{ghys01}.  Given a representation $\rho: G \rightarrow  \mathrm{Homeo}_+(S^1)$, we can explicitly describe a representative $\omega:G^2 \rightarrow \mathbb{Z}$ of the bounded Euler class $eu(\rho)$ as follows.  For each $g \in G$, all the choices of lifts $\widetilde{\rho(g)} : \mathbb{R} \rightarrow \mathbb{R}$ differ by an integral translation, so we can choose for each $g \in G$ a lift satisfying $\widetilde{\rho(g)}(0) \in [0,1)$.  Then a bounded representative of $eu(\rho)$ is given by:
\[ \omega(g,h) = \widetilde{\rho(g)}(\widetilde{\rho(h)}(0)) - \widetilde{\rho(gh)}(0), 
\]
which is an element of $\mathbb{Z}$ (see, e.g. \cite[Lemma 6.3]{ghys01}).

\begin{prop} If $G$ is a countable group and $f$ is a circular ordering of $G$, then $[f] = eu(\rho_f) \in H^2_b(G; \mathbb{R})$.
\end{prop}
\begin{proof}
Let  $\tilde{t}: \widetilde{G}_f \rightarrow \mathbb{R}$ denote the tight embedding used in constructing the dynamic realisation $\rho_f : G \rightarrow \mathrm{Homeo}_+(S^1)$, recall that $\tilde{t}$ satisfies $\tilde{t}(z_f^k) = k$ and is order-preserving with respect to the orderings $<_f$ of $\widetilde{G}_f$ and $<$ of $\mathbb{R}$.    

For each $g \in G$ our choice of lift $\widetilde{\rho_f(g)}$ used to compute the bounded Euler class will be $\widetilde{\rho}_f(\widetilde{g})$, recall that $\widetilde{g} = (g,0)$.  Note that it satisfies $\widetilde{\rho}_f(\widetilde{g})(0) = t(g,0)$, and since $(id, 0) \leq_f (g, 0) <_f z_f$ we have $0 \leq t(g, 0) < 1$.  Then with these choices we are able to compute the following representative of $eu(\rho_f)$:
\[ \omega(g,h)= \widetilde{\rho}_f(\widetilde{g})(\widetilde{\rho}_f(\widetilde{h})(\tilde{t}(1))) - \widetilde{\rho}_f(\widetilde{gh})(\tilde{t}(1)) = \tilde{t}(\widetilde{g} \widetilde{h}) - \tilde{t}(\widetilde{gh}). 
\]
Since $\widetilde{g} \widetilde{h} = (g, 0) (h, g) = (gh, f(g,h))$ then $ \tilde{t}(\widetilde{g} \widetilde{h}) = \tilde{t}(\widetilde{gh} z_f^{f(g,h)}) = \tilde{t}(\widetilde{gh}) + f(g,h)$. This yields $\omega(g,h) = f(g,h)$.
\end{proof}

\begin{prop}
\label{semiconj of orders}
Suppose $f_1$ and $f_2$ are two circular orderings on a countable group $G$ such that the dynamic realisations $\rho_{f_1}, \rho_{f_2}$ are semiconjugate. Then $[f_1] = [f_2] \in H^2_b(G;\ZZ)$.
\end{prop}
\begin{proof}
This follows immediately from the fact that $eu(\rho_{f_1}) = eu(\rho_{f_2})$ whenever $\rho_{f_1}$ and $\rho_{f_2}$ are semiconjugate (\cite[Theorem 6.6]{ghys01}, or \cite{ghys84}).
\end{proof}

%Owing to this Proposition, we can say that \emph{two circular orderings $f_1, f_2$ are semiconjugate} if and only if $[f_1] = [f_2] \in H^2_b(G;\ZZ)$, which happens if and only if $\rho_{f_1}, \rho_{f_2}$ are semi-conjugate in the classical sense.

%\tg{Remark here something about if you start with tight embeddings you actually get conjugate actions (not just semi-conjugate), but it's not really relevant to the rest of the paper. I vaguely remember talking about this on 13th Sept, and I can't remember if this is something we wanted to do or not.}

\subsection{Rotation and translation numbers, algebraically and dynamically}
This section prepares the necessary notation to discuss fractional Dehn twist coefficients from both dynamical 
 and algebraic perspectives.
 %, but also introduces several algebraic definitions of classical objects required to connect our algebraic arguments to their dynamical counterparts.  
 Background on the classic dynamical development of these ideas can be found in \cite{ghys01, Herman79}, the algebraic ideas appear also in \cite{BC21}.

If $G \subset \mathrm{H}\widetilde{\mathrm{ome}}\mathrm{o}_+(S^1)$, we define the dynamical translation number (i.e., the classical translation number due to Poincar\'{e} \cite{H1885}) of an element $g \in G$  as 
\[
\tau^D(g) = \lim_{n \to \infty} \frac{g^n(x) - x}{n}, 
\]
where $x \in \mathbb{R}$ is arbitrary. The limit exists and is independent of $x \in \mathbb{R}$ \cite[Proposition 2.3]{Herman79}.  
%From here we arrive at rotation numbers as follows: Recall that
Using $q:  \mathrm{H}\widetilde{\mathrm{ome}}\mathrm{o}_+(S^1) \rightarrow \mathrm{Homeo}_+(S^1)$ to denote the quotient map, if $G \subset \mathrm{Homeo}_+(S^1)$ then write $\widetilde{g} \in q^{-1}(g)$ to denote an arbitrary choice of preimage.  Then we define the dynamic rotation number of $g$ to be 
\[\mathrm{rot}^D(g) = \tau^D(\widetilde{g}) \mod \mathbb{Z}.
\]
This definition is independent of the choice of $\widetilde{g}$.

On the other hand, these definitions can also be described algebraically.  Given a left-ordered group $(G, <)$ and a positive, cofinal central element $z \in G$,  we define the floor of $g$ relative to $<$ to be the unique integer $[g]_<$ such that $z^{[g]_<} \leq g < z^{[g]_< +1}$.  Then for every $g \in G$, we can define the  \emph{algebraic translation number} of $g \in G$ relative to $<$ to be 
\[ \tau^A_<(g) = \lim_{n \to \infty}\frac{[g^n]_<}{n}.
\]
This limit always exists as the sequence $\{ [g^n]_< \}_{n \geq 0}$ satisfies $[g^n]_<+[g^m]_< \leq [g^{m+n}]_<$ and is therefore superadditive, so we may apply Fekete's lemma. 

There is a special circumstance where we have already seen that left orderings with cofinal central elements arise naturally.  Given a circular ordering $f$ of a group $G$, recall that the left-ordered lift $(\widetilde{G}_f, <_f)$ comes equipped with a positive, cofinal central element $z_f$.  To simplify notation, in this setting we will write $[g]_f$ in place of $[g]_{<_f}$ and $\tau^A_f(g)$ in place of $\tau^A_{f_<}(g)$ for all $g \in \widetilde{G}_f$.  Then for every $g \in G$ we define the \emph{algebraic rotation number} of $g$ to be 
\[ \mathrm{rot}_f^A(g) = \tau^A_f(g, k) \mod \mathbb{Z}
\] 
where $(g, k) \in \widetilde{G}_f$ and $k \in \mathbb{Z}$ is arbitrary.  We observe that this definition is independent of $k$ by noting that $(g,k)^n = (g^n, \sum_{i=1}^{n-1}f(g^i, g) + nk)$, so that $[(g,k)^n]_f = \sum_{i=1}^{n-1}f(g^i, g) + nk$.  It follows that for $k, k' \in \mathbb{Z}$
we have $ \tau^A_f(g, k) -\tau^A_f(g,k') = k-k'$, so that $\mathrm{rot}_f^A(g)$ is well-defined.

%We connect the notion of algebraic and dynamical rotation and translation numbers via the dynamical realization, as follows.  Suppose that $G$ is countable and equipped with a circular ordering $f$.  Let $\widetilde{\rho}_f : \widetilde{G}_f \rightarrow \mathrm{H}\widetilde{\mathrm{ome}}\mathrm{o}_+(S^1)$ denote the dynamic realisation of $(\widetilde{G}_f, <_f)$, where we assume that $\widetilde{\rho}_f(z_f)(x) = x+1$ for all $x \in \mathbb{R}$.  Define the \emph{dynamic translation number} of $g \in \widetilde{G}_f$ to be 
%\[
%\tau^D_f(g) = \tau^D(\widetilde{\rho}_f(g))\] For $g \in G$, define the \emph{dynamic rotation number} of $g$ to be 
%\[\mathrm{rot}_f^D(g) = \mathrm{rot}^D(\rho_f(g)).
%\]

\begin{prop}
\label{order to rep}
Let $(G,<)$ be a countable left-ordered group with positive, cofinal central element $z$, and $\rho : G \rightarrow  \mathrm{H}\widetilde{\mathrm{ome}}\mathrm{o}_+(S^1)$ a dynamic realisation of $<$ satisfying $\rho(z)(x) = x+1$ for all $x \in \mathbb{R}$.  Then $\tau^A_<(g) = \tau^D(\rho(g))$ for all $g \in G$.
\end{prop}
\begin{proof}
Let $t :G \rightarrow \mathbb{R}$ denote the tight embedding used to construct $\rho$. Observe that for all $g \in G$ and for all $n \in \mathbb{Z}$ we have $\rho(g^n)(0) =\rho(g^n)(t(id)) = t(g^n)$.

Then as $t(z^k) = k$, note that $[g]_< = \lfloor t(g) \rfloor$.  Consequently
\[ \tau^A_<(g) = \lim_{n \to \infty} \frac{[g^n]_<}{n} = \lim_{n \to \infty} \frac{\lfloor \rho(g^n)(0) \rfloor}{n}  =  \lim_{n \to \infty} \frac{ \rho(g^n)(0) }{n} =  \tau^D(\rho(g)). \hfill \qedhere
\]

\end{proof}

Proposition \ref{order to rep} begins with an ordered group and shows that the canonical representation corresponding to the order allows one to recover the algebraic translation numbers from the dynamics of the action.  On the other hand, the next proposition starts with a representation and builds an ordering of $G$ whose algebraic translation numbers agree with those arising from the given dynamics.

\begin{prop}
\label{rep to order}
Suppose that $\rho:G \rightarrow  \mathrm{H}\widetilde{\mathrm{ome}}\mathrm{o}_+(S^1)$ is an injective homomorphism and that $z \in G$ satisfies $\rho(z)(x) = x+1$ for all $x \in \mathbb{R}$.  Fix a left ordering $\prec$ of $G$, and define a new left ordering $<$ of $G$ according to the rule $g<h$ if and only if $\rho(g)(0) < \rho(h)(0)$ or $\rho(g)(0) = \rho(h)(0)$ and $g \prec h$.  Then $z$ is positive and cofinal relative to the ordering $<$ of $G$, and $\tau^D(\rho(g)) = \tau^A_<(g)$ for all $g \in G$.
\end{prop}
\begin{proof}
That $z$ is positive and $<$-cofinal follows from the definition of $<$.  Next, note that $\lfloor \rho(g)(0) \rfloor = [g]_<$ for all $g \in G$.  Therefore
\[ \tau^D(\rho(g)) = \lim_{n \to \infty} \frac{\rho(g^n)(0)}{n} =  \lim_{n \to \infty} \frac{\lfloor \rho(g^n)(0) \rfloor }{n} = \lim_{n \to \infty} \frac{[g^n]_<}{n} = \tau^A_<(g). \hfill \qedhere
\]
\end{proof}

\begin{cor}
If $(G, f)$ is a countable circularly-ordered group and $\rho_f: G \rightarrow \mathrm{Homeo}_+(S^1)$ the corresponding dynamic realisation, then for every $g \in G$, $\mathrm{rot}^A_f(g) = \mathrm{rot}^D(\rho_f(g))$.
\end{cor}

Last, we observe that these quantities are conjugation-invariant, which is an easy observation from the classical definitions.  We highlight this fact here as it will be needed in the proof of Theorem \ref{main theorem}.

\begin{prop}
\label{conjugacy invariance}
If $(G, f)$ is a  circularly-ordered group, then $\tau^A_f(g) = \tau^A_f(hgh^{-1})$ for all $g ,h \in \widetilde{G}_f$ and $\mathrm{rot}^A_f(g) = \mathrm{rot}^A_f(hgh^{-1})$ for all $g, h \in G$.
\end{prop}
\begin{proof}
From $z_f^{[g^n]_f} \leq g^n < z_f^{[g^n]_f +1}$ and $z_f^{[h]_f} \leq h < z_f^{[h]_f +1}$ one finds $z_f^{[g^n]_f-1} \leq hg^nh^{-1} < z_f^{[g^n]_f +2}$, so that $[g^n]_f-1 \leq [(hgh^{-1})^n]_f \leq [g^n]_f+1$.  It follows that 
\[\tau^A_f(g) = \lim_{n \to \infty}\frac{[g^n]_f}{n} = \lim_{n \to \infty}\frac{[(hgh^{-1})^n]_f}{n} = \tau^A_f(hgh^{-1}).
\]
It follows that rotation number is also conjugation-invariant.
\end{proof}

\section{Mapping class groups}\label{sec:mcg-basics}

In this section we will recall some useful facts about mapping class groups and prove Lemma \ref{lem:modaut-inner}, a key step in the proof of Theorem \ref{main theorem}. For an introduction to mapping class groups, see \cite{FM12}.

Let $\Sigma_{g,n}^b$ be a compact orientable surface of genus $g$ with $b$ boundary components and $n$ marked points disjoint from $\partial\Sigma_{g,n}^b$. Denote the set of marked points by $\vP$. If $n$ or $b$ is 0 we will omit the subscript or superscript. Let $\Homeo^+(\Sigma_{g,n}^b,\vP,\partial\Sigma_{g,n}^b)$ be the set of orientation-preserving homeomorphisms $f$ of $\Sigma_{g,n}^b$ so that $f(\vP) = \vP$ and $f|_{\partial \Sigma_{g,n}^b} = \text{Id}$. Note that if $\partial\Sigma_{g,n}^b \neq \emptyset$, then all homeomorphisms fixing the boundary pointwise are orientation-preserving.

The \emph{mapping class group} is the quotient group 
\[
\Mod(\Sigma_{g,n}^b) = \Homeo^+(\Sigma_{g,n}^b,\vP,\partial\Sigma_{g,n}^b)/\Homeo_0(\Sigma_{g,n}^b,\vP,\partial\Sigma_{g,n}^b),
\]
where $\Homeo_0(\Sigma_{g,n}^b,\vP,\partial\Sigma_{g,n}^b)$ is the subgroup of homeomorphisms isotopic to the identity. The isotopies must be via elements of $\Homeo^+(\Sigma_{g,n}^b,\vP,\partial\Sigma_{g,n}^b)$. Elements of the mapping class group are referred to as \emph{mapping classes}.

Let $\alpha$ be a simple closed curve disjoint from $\partial  \Sigma_{g,n}^b$ and $\vP$. Choose a regular neighbourhood of $\alpha$ that is homeomorphic, via an orientation-preserving homeomorphism, to an annulus $A \simeq S^1 \times I$ disjoint from $\partial  \Sigma_{g,n}^b$ and $\vP$. Define the \emph{Dehn twist about $\alpha$}, denoted $T_\alpha$, to be the homeomorphism of $ \Sigma_{g,n}^b$ given by $(s,t) \mapsto (se^{-2\pi i t},t)$ on $A$, and the identity outside of $A$.

Once an orientation on $ \Sigma_{g,n}^b$ has been fixed, the isotopy class of a Dehn twist $T_\alpha$ depends only on the unoriented isotopy class of $\alpha$. Therefore if $a$ is the unoriented isotopy class of $\alpha$, we may abuse notation and write the mapping class $[T_\alpha] \in \Mod(\Sigma_{g,n}^b)$ as $T_\alpha$ or $T_a$. It is important to note that Dehn twists have infinite order \cite[Proposition 3.2]{FM12} and if $a$ is isotopic to a boundary component, then $T_a$ is central in $\Mod(\Sigma_{g,n}^b)$ \cite[Fact 3.8]{FM12}. 

\subsection{Inner automorphisms and capping}
By gluing on a disk with one marked point to the boundary of $\Sigma_g^1$ we obtain $\Sigma_{g,1}$. By extending homeomorphisms of $\Sigma_g^1$ by the identity on the marked disk, we obtain the central extension
\[
1 \lra \langle T_d\rangle \lra \Mod(\Sigma_g^1) \lra \Mod(\Sigma_{g,1}) \lra 1,
\]
where $d$ is the isotopy class of the boundary curve. The process of obtaining $\Sigma_{g,1}$ from $\Sigma_g^1$ like this is commonly called \emph{capping the boundary}, and the surjective map $\Mod(\Sigma_g^1) \to \Mod(\Sigma_{g,1})$ is the \emph{capping homomorphism}. Note that a similar capping homomorphism exists for any surface with at least 1 boundary component (see \cite[Section 3.6.2]{FM12}), but we will not need the full generality. 

Before we prove Lemma \ref{lem:modaut-inner}, we will need the following general lemma about automorphisms of central extensions.

\begin{lem} \label{lem:extension-automorphism}
Let $1 \to A \hookrightarrow G \to H \to 1$ be a central extension of $H$ by $A$. Let $\phi,\psi \in \Aut(G)$ be such that $\phi(A) = \psi(A) = A$, and such that the induced automorphisms $\ol{\phi},\ol{\psi} \in \Aut(H)$ satisfy $\ol{\phi} = \ol{\psi}$. Then $\eta:G \to A$ given by $\eta(g) = \phi(g)\psi(g)^{-1}$ is a homomorphism.
\end{lem}
\begin{proof}
Note that since $\ol{\phi} = \ol{\psi}$, $\eta(g)$ is indeed an element of $A$. For $g,h \in G$ we have
\[
\eta(gh) = \phi(gh)\psi(gh)^{-1} = \phi(g)\phi(h)\psi(h)^{-1}\psi(g)^{-1} = \phi(g)\eta(h)\psi(g)^{-1} = \phi(g)\psi(g)^{-1}\eta(h) = \eta(g)\eta(h),
\]
completing the proof.
\end{proof}

For two isotopy classes $a,b$ of simple closed curves on a surface, denote the \emph{geometric intersection number} by $i(a,b)$, that is, the minimum number of intersection points between any representatives of $a$ and $b$. It is a useful fact that $T_aT_bT_a = T_bT_aT_b$ if and only if $i(a,b) = 1$ \cite[Propositions 3.11 and 3.13]{FM12}.

A \emph{$k$-chain} on a surface is a $k$-tuple $(a_1,\ldots,a_k)$ of isotopy classes of simple closed curves such that $i(a_j,a_{j+1}) = 1$ for all $j \in \{1,\ldots,k-1\}$, and $i(a_j,a_l) = 0$ otherwise.  Choose representatives $\alpha_i$ of $a_i$ so that the $\alpha_i$ are in minimal position. If $k$ is even, a regular neighbourhood of $\cup_{i=1}^k \alpha_i$ is homeomorphic to a genus $\frac k2$ surface with 1 boundary component. Let $e$ be the isotopy class of the boundary component. The relation
\[
(T_{a_1}T_{a_2}\cdots T_{a_k})^{2k+2} = T_e
\]
holds, and is known as the \emph{chain relation} \cite[Proposition 4.12]{FM12}. It follows that if $(a_1,\ldots,a_{2g})$ is any $2g$-chain on $\Sigma_g^1$, 
\[
(T_{a_1}T_{a_2}\cdots T_{a_{2g}})^{4g+2} = T_d
\]
where $d$ is the isotopy class of the boundary of $\Sigma_g^1$.

Before embarking on the proof of Lemma \ref{lem:modaut-inner}, we must recall an important result due to Ivanov (\cite[Theorem 2]{Ivanov97}, see also \cite{Ivanov88}). Let $g \geq 2, n \geq 1$, and let $\Mod^{\pm}(\Sigma_{g,n})$ denote the \emph{extended mapping class group} of $\Sigma_{g,n}$ (that is, the mapping class group where we allow orientation-reversing homeomorphisms).  Note that $\Mod(\Sigma_{g,n})$ is an index-2 subgroup of $\Mod^{\pm}(\Sigma_{g,n})$. Ivanov's theorem states that the map $\Mod^{\pm}(\Sigma_{g,n}) \to \Aut(\Mod(\Sigma_{g,n}))$ given by $\gamma \mapsto (f \mapsto \gamma f \gamma^{-1})$ is an isomorphism. It follows that for all isotopy classes of simple closed curves $a$ on $\Sigma_{g,n}$, if $\gamma \in \Mod(\Sigma_{g,n})$, then $\gamma T_a \gamma^{-1} = T_{\gamma(a)}$, and if $\gamma \notin \Mod(\Sigma_{g,n})$, then $\gamma T_a \gamma^{-1} = T_{\gamma(a)}^{-1}$. In particular, we can identify whether or not an automorphism of $\Mod(\Sigma_{g,n})$ is inner by simply observing whether a Dehn twist is sent to a Dehn twist, or the inverse of a Dehn twist.

In preparation for the next lemma, let $\vC:\Mod(\Sigma_g^1) \to \Mod(\Sigma_{g,1})$ be the capping homomorphism. For each isotopy class $a$ of a simple closed curve on $\Sigma_g^1$, let $\hat a$ be the isotopy class that is the image of $a$ under the inclusion $\Sigma_g^1 \hookrightarrow \Sigma_{g,1}$. Note that every isotopy class of a simple closed curve on $\Sigma_{g,1}$ is of the form $\hat a$ for some isotopy class $a$ on $\Sigma_g^1$.

\begin{lem} \label{lem:modaut-inner}
Let $g\geq 2$, and let $\varphi\in \Aut(\Mod(\Sigma_g^1))$ be such that $\varphi(T_d) = T_d$, where $d$ is the isotopy class of the boundary curve. Then $\varphi$ is an inner automorphism.
\end{lem}
\begin{proof}
Since $\varphi(T_d) = T_d$, we have an induced automorphism $\ol{\varphi} \in \Aut(\Mod(\Sigma_{g,1}))$. Then by \cite[Theorem 2]{Ivanov97}, there exists $\epsilon \in \{\pm 1\}$ so that for all isotopy classes of simple closed curves $\hat a$ on $\Sigma_{g,1}$, $\ol{\varphi}(T_{\hat a}) = T_{\hat b}^{\epsilon}$ where $\hat{b}$ is the image of $\hat{a}$ under an appropriately chosen element of $\Mod^{\pm}(\Sigma_{g,1})$. We will first show that $\epsilon = 1$, with the aim of concluding that $\ol{\varphi}$ is an inner automorphism.

Let $\{\hat{a}_1,\ldots,\hat{a}_{2g}\}$ be a $2g$-chain and suppose that for each $i \in \{1,\ldots,2g\}$, $\ol{\varphi}(T_{\hat{a}_i}) = T_{\hat{b}_i}^\epsilon$. Then $\varphi(T_{a_i}) = T_d^{s_i}T_{b_i}^\epsilon$ for some $s_i \in \ZZ$. Note $\{b_1,\ldots,b_{2g}\}$ is a $2g$-chain. In particular, for $i < 2g$, we have $i(b_i,b_{i+1}) = 1$ so 
\begin{align*}
1 &= \varphi(T_{a_i}T_{a_{i+1}}T_{a_i}T_{a_{i+1}}^{-1}T_{a_i}^{-1}T_{a_{i+1}}^{-1}) \\
&= T_d^{s_i - s_{i+1}}T_{b_i}^\epsilon T_{b_{i+1}}^\epsilon T_{b_i}^\epsilon T_{b_{i+1}}^{-\epsilon}T_{b_i}^{-\epsilon}T_{b_{i+1}}^{-\epsilon} \\
&= T_d^{s_i - s_{i+1}}.
\end{align*}
Therefore $s_i = s_{i+1}$ for all $i < 2g$. Let $s = s_i$. The chain relation gives
\begin{align*}
T_d   = \varphi(T_d) &= \varphi((T_{a_1}\cdots T_{a_{2g}})^{4g+2}) \\
&= T_d^{s(2g)(4g+2)} (T_{b_1}^\epsilon \cdots T_{b_{2g}}^\epsilon)^{4g+2} \\
&= T_d^{s(2g)(4g+2)}T_d^\epsilon.
\end{align*}
Therefore $ 1 = 2gs(4g+2)+\epsilon$. Since $g\geq 2$, we must have $s = 0$ and $\epsilon = 1$. Therefore $\ol{\varphi}$ is an inner automorphism.

Let $\nu \in \Mod(\Sigma_{g,1})$ be such that $\ol{\varphi}(f) = \nu f \nu^{-1}$ for all $f \in \Mod(\Sigma_{g,1})$, and choose $\wt \nu \in \vC^{-1}(\nu)$. Let $\theta \in \Aut(\Mod(\Sigma_g^1))$ be the inner automorphism given by conjugation by $\wt \nu$, that is $\theta(f) = \wt \nu f \wt \nu^{-1}$ for all $f \in \Mod(\Sigma_g^1)$. Then $\theta(T_d) = T_d$ and $\ol{\theta} = \ol{\varphi}$. Now, every homomorphism $\Mod(\Sigma_g^1) \to \ZZ$ is trivial \cite[Theorem 5.2]{FM12}. Therefore by Lemma \ref{lem:extension-automorphism}, $\varphi(f)\theta(f)^{-1} = 1$ for all $f \in \Mod(\Sigma_g^1)$. We may now conclude $\varphi(f) = \theta(f)$, and so $\varphi(f)$ is an inner automorphism.
\end{proof}

\section{Cofinality of boundary Dehn twists}\label{sec:cofinal}

The goal of this section is to prove that in the mapping class group of $\Sigma_g^1$, an orientable surface of genus $g$ with 1 boundary component, the Dehn twist about a curve isotopic to the boundary component is $<$-cofinal for {\it every} left ordering on $\Mod(\Sigma_g^1)$ (Theorem \ref{thm:cofinal-boundary-twist}).  This result will imply that Theorem \ref{main theorem} applies to all the actions of $\Mod(\Sigma_g^1)$ on $\mathbb{R}$, up to conjugation (Proposition \ref{conj to shift}).

We begin with a general lemma concerning left-orderable groups.

\begin{lem} \label{lem:generated-by-roots}
Let $G$ be a left-orderable group, and $z \in G$ a central element. Suppose there is a generating set $\{g_i\}_{i \in I}$ such that for each $i \in I$, there exist $n_i, m_i \in \ZZ \setminus \{0\}$ such that $g_i^{n_i} = z^{m_i}$. Then $z$ is $<$-cofinal for every left ordering $<$ of $G$.
\end{lem}
\begin{proof}
Consider the set $H = \{g \in G \mid \exists n \in \ZZ \text{ such that } z^{-n} < g < z^n \}$. It suffices to show $H = G$. We first show that $H$ is a subgroup. Note $z \neq 1$ since some power of $z$ is a power of every generator, and $G$ is torsion free. Therefore $z < 1 < z^{-1}$ or $z^{-1} < 1 < z$, so $1 \in H$. Next, if $z^{-n} < g < z^n$ for some $g \in G$ then $z^n < g^{-1} < z^{-n}$. Finally, if $z^{-n} < g < z^n$ and $z^{-m} < h < z^m$, then $gh > gz^{-m} = z^{-m}g > z^{-m}z^{-n} = z^{-(m+n)}$ and similarly $gh < z^{m+n}$. Thus $H$ is a subgroup of $G$.

By possibly replacing $z$ and each $g_i$ with its inverse, we may assume $z > 1$ and $g_i > 1$ for all $i \in I$, and that $n_i, m_i > 0$. Then
\[
z^{-m_i-1} < 1 < g_i < g_i^{n_i} = z^{m_i} < z^{m_i+1}.
\]
Therefore $H$ contains a generating set for $G$, so that $H = G$. 
\end{proof}

Our goal now is to apply Lemma \ref{lem:generated-by-roots} to the mapping class group of a genus $g> 0$ surface with 1 boundary component, $\Sigma_g^1$, proving Theorem \ref{thm:cofinal-boundary-twist}.

%\begin{thm} \label{thm:cofinal-boundary-twist}
%Let $T_d$ be the Dehn twist about a curve isotopic to the boundary of $\Sigma_g^1$. Then $T_d$ is cofinal in every left ordering of $\Mod(\Sigma_g^1)$.
%\end{thm}

\begin{proof}[Proof of Theorem \ref{thm:cofinal-boundary-twist}]
Let $(a_1,\ldots,a_{2g})$ be a $2g$-chain on $\Sigma_g^1$. Let $X = T_{a_1}T_{a_2}\cdots T_{a_{2g}}$ and $Y = T_{a_1}^2T_{a_2}\cdots T_{a_{2g}}$. Then $X^{4g+2} = Y^{4g} = T_d$ \cite[Section 4.4.1]{FM12}. Since $T_d$ is central in $\Mod(\Sigma_g^1)$, all conjugates of $X$ and $Y$ are roots of $T_d$. Note that $YX^{-1} = T_{a_1}$ and $a_1$ is a non-separating simple closed curve. All Dehn twists about non-separating simple closed curves are conjugate \cite[Section 1.3.1 and Fact 3.8]{FM12} and $\Mod(\Sigma_g^1)$ is generated by Dehn twists about non-separating simple closed curves \cite[Chapter 4]{FM12}. Therefore $\Mod(\Sigma_g^1)$ is generated by conjugates of $X$ and $Y$. The proof concludes by applying Lemma \ref{lem:generated-by-roots}.
\end{proof}

By considering $2g+1$-chains on $\Sigma_{g}^2$ and applying the chain relations as in the above proof, one obtains the result that for surfaces $\Sigma_g^2$, the product of the Dehn twists about curves isotopic to the boundary components is cofinal and central in every left ordering of $\Mod(\Sigma_g^2)$. We conjecture something stronger is true.

\begin{conj}
Let $g \geq 2$ and let $b_1,\ldots,b_n$ be curves isotopic to the boundary components of $\Sigma_g^n$. Any element of the form $\Pi_{i=1}^n T_{b_i}^{k_i}$ for any positive exponents $k_1,\ldots,k_n$ is cofinal in every left ordering of $\Mod(\Sigma_g^n)$.
\end{conj}

%\tg{If we can figure it out, we change that to product of all boundary Dehn twists.}

\section{Fractional Dehn twist coefficients and actions on \texorpdfstring{$\mathbb{R}$}{R}}
\label{FDTC and orders}

\subsection{Fractional Dehn twist coefficients}
\label{FDTC}

Recall that if $\Sigma_{g,n}^b$ is a hyperbolic surface with $b>0$, there is a ``standard action" of $\Mod(\Sigma_{g,n}^b)$ on $\mathbb{R}$ that is constructed as follows.

First, we construct the universal cover $p: \widetilde{\Sigma_{g,n}^b} \rightarrow \Sigma_{g,n}^b$, and note that we can think of $\widetilde{\Sigma_{g,n}^b}$ as a closed subset of $\mathbb{H}^2$.   Fix a point $x_0 \in \partial \Sigma_{g,n}^b$, say in a component $C$ of the boundary, and a point $\tilde{x_0} \in \widetilde{C} \subset \partial \widetilde{\Sigma_{g,n}^b}$ with $p(\tilde{x_0}) = x_0$.  Now for each $h \in \Mod(\Sigma_{g,n}^b)$, there is a unique lift of $h$ satisfying $h(\tilde{x_0}) = \tilde{x_0}$ yielding an action of $\Mod(\Sigma_{g,n}^b)$ on $\widetilde{\Sigma_{g,n}^b}$ fixing $\tilde{x_0}$ and thus fixing $\widetilde{C}$.

 Now we can identify $\partial \widetilde{\Sigma_{g,n}^b} \setminus \ell$, where $\ell \subset \overline{\mathbb{H}^2}$ is the closure of $\widetilde{C}$, with the interval $(0,\pi)$, and thus with $\mathbb{R}$, by identifying each point $y$ on the boundary with the unique geodesic from $x_0$ to $y$.  Then observe that the action of $\Mod(\Sigma_{g,n}^b)$ extends to an action on $\partial \widetilde{\Sigma_{g,n}^b}$ by orientation-preserving homeomorphisms, which is homeomorphic to $\mathbb{R}$.  We orient the boundary and parameterise it so that the action of the boundary Dehn twist $T_C$ satisfies $T_C(x) = x+1$ for all $x \in \mathbb{R}$.  This defines a representation 
 \[ \rho_{s,C}: \Mod(\Sigma_{g,n}^b) \rightarrow \mathrm{H}\widetilde{\mathrm{ome}}\mathrm{o}_+(S^1)
 \]
which we call the \emph{standard representation with respect to $C$}.  The \emph{fractional Dehn twist coefficient} of $h \in \Mod(\Sigma_{g,n}^b)$ can be defined as
\[ c(h, C) = \tau^D(\rho_{s,C}(h)).
\]
While this is not the usual definition of the fractional Dehn twist coefficient, that this is equivalent to it appears in \cite[Theorem 4.16]{IK17}, and for the special case of $\Mod(\Sigma_{0,n}^1)$ (i.e. for the braid groups) in \cite{Malyutin04}.  When $b=1$, we will simplify our notation and use $ \rho_{s}$ to denote the standard representation, and $ c(h)$ to denote the fractional Dehn twist coefficient.

\subsection{Actions of \texorpdfstring{$\Mod(\Sigma_{g,n}^b)$}{Sgn} on \texorpdfstring{$\mathbb{R}$}{R}}

In this section we prove Theorem \ref{conj to shift}, which is Theorem \ref{conj to shift intro version} from the introduction, from which Theorem \ref{fdtc from orders} follows.  We begin with a preparatory lemma.

\begin{lem}\label{translation same up to isom}
    Suppose $<_1$ and $<_2$ are left orderings of a group $G$ such that there exists a central element $z \in G$ that is $<_i$-cofinal for both $i = 1,2$. Suppose further that $[f_{<_1}] = [f_{<_2}] \in H^2_b(G/\langle z \rangle ; \ZZ)$. Then there exists an automorphism $\phi$ of $G$ satisfying $\phi(z) = z$ and $\tau_{<_1}^A(g) = \tau_{<_2}^A(\phi(g))$ for all $g \in G$.
\end{lem}
\begin{proof}
    To ease notation, we will simply write $f_i = f_{<_i}$ and $[g]_{<_i} = [g]_i$ for $i = 1,2$. Let $\{g\}_i$ be the unique coset representative of $g\langle z \rangle$ so that $id \leq_i \{g\}_i <_i z$. Note that every $g \in G$ is uniquely written as $\{g\}_iz^{a_i}$, and in this case $a_i = [g]_i$. Recall that $\{g\}_i\{h\}_i = \{gh\}_iz^{f_i(g\langle z \rangle,h\langle z \rangle)}$, and so $[gh]_i = [g]_i + [h]_i + f_i(g\langle z \rangle,h\langle z \rangle)$.  

    Since $[f_1] = [f_2] \in H_b^2(G/\langle z \rangle ;\ZZ)$, there is a bounded function $d:G/\langle z \rangle \to \ZZ$ such that for all $g, h \in G$,
    \[
    f_{1}(g\langle z \rangle,h\langle z \rangle) - f_{2}(g\langle z \rangle,h\langle z \rangle) = d(g\langle z \rangle) - d(gh\langle z \rangle) + d(h \langle z \rangle).
    \]
    Define $\phi:G \to G$ by $\phi(g) = \{g\}_2z^{[g]_1 + d(g \langle z \rangle)}$. We have
    \begin{align*}
        \phi(g)\phi(h) &= \{g\}_2z^{[g]_1 + d(g \langle z \rangle)}\{h\}_2 z^{[h]_1 + d(h \langle z \rangle)} \\
        &= \{gh\}_2z^{[g]_1 + [h]_1 + d(g\langle z \rangle) + d(h \langle z \rangle) + f_2(g\langle z \rangle,h\langle z \rangle)} \\
        &= \{gh\}_2z^{[g]_1 + [h]_1 + d(gh\langle z \rangle) + f_1(g\langle z \rangle,h\langle z \rangle)} \\
        &= \{gh\}_2z^{[gh]_1 + d(gh\langle z \rangle)} \\
        &= \phi(gh).
    \end{align*}
    so $\phi$ is a homomorphism. An inverse is given by $\phi^{-1}(g) = \{g\}_1z^{[g]_2 - d(g\langle z \rangle)}$ so $\phi$ is an automorphism of $G$. We also check that
    \begin{align*}
        \tau_{<_2}^A(\phi(g))  - \tau_{<_1}^A(g) &= \lim_{n \to \infty} \frac{[\phi(g^n)]_2 - [g^n]_1}{n} \\
        &= \lim_{n \to \infty} \frac{[g^n]_1 + d(g^n \langle z \rangle) - [g^n]_1}{n} \\
        &= \lim_{n \to \infty} \frac{d(g^n\langle z \rangle)}{n}\\
        & = 0,
    \end{align*}
    where the last equality follows since $d$ is bounded. Finally, observe that $d(\langle z \rangle) = 0$, and for $i = 1,2$, $\{z\}_i = id$ and $[z]_i = 1$. Thus $\phi(z) = z$.
\end{proof}

For the statement and proof of the next theorem, recall that $T_d$ denotes the Dehn twist around a simple closed curve $d$ that is parallel to $\partial \Sigma_g^1$.

\begin{thm}
\label{main theorem}
Suppose that $g \geq 2$ and that $\rho_i: \Mod(\Sigma_g^1) \rightarrow \mathrm{H}\widetilde{\mathrm{ome}}\mathrm{o}_+(S^1)$ is an injective homomorphism satisfying $\rho_i(T_d)(x) = x+1$ for $i =1,2$ and for all $x \in \mathbb{R}$.   Then $ \tau^D(\rho_1(h)) = \tau^D(\rho_2(h))$ for all $h \in \Mod(\Sigma_g^1) $.
\end{thm}
\begin{proof}
    Fix a left ordering $\prec$ of $\Mod(\Sigma_g^1)$.  Then, associated to each homomorphism $\rho_i$, there is an ordering $<_i$ defined as in Proposition \ref{rep to order} that  satisfies $\tau^D(\rho_i(h)) = \tau^A_{<_i}(h)$ for $i = 1, 2$ and for all $h \in  \Mod(\Sigma_g^1)$. By Theorem \ref{thm:cofinal-boundary-twist}, $T_d$ is $<_i$-cofinal for $i = 1,2$. Consider the dynamic realisations of the circular orderings $f_{<_i}$ on $\Mod(\Sigma_g^1)/\langle T_d \rangle = \Mod(\Sigma_{g,1})$. By \cite{MW}, the dynamic realisations are semiconjugate, so $[f_{<_1}] = [f_{<_2}] \in H^2_b(\Mod(\Sigma_{g,1};\ZZ)$ by Proposition \ref{semiconj of orders}. Therefore by Lemma \ref{translation same up to isom}, there is an automorphism $\phi$ of $\Mod(\Sigma_g^1)$ so that $\phi(T_d) = T_d$ and $\tau_{<_1}^A(h) = \tau_{<_2}^A(\phi(h))$ for all $h \in \Mod(\Sigma_g^1)$. By Lemma \ref{lem:modaut-inner}, $\phi$ is an inner automorphism. Thus by Proposition \ref{conjugacy invariance}, $\tau_{<_1}^A(h) = \tau_{<_2}^A(h)$, and the proof is complete.
\end{proof}

For the next proof, recall that a subgroup $C$ of a left-ordered group $(G, <)$ is called \emph{convex} if, whenever $c, d \in C$ and $g \in G$ then $c<g<d$ implies $g \in C$.  The next theorem is Theorem \ref{conj to shift intro version} from the introduction.

\begin{thm}
\label{conj to shift} Suppose that $g \geq 2$ and let $\rho:\Mod(\Sigma_g^1) \rightarrow \mathrm{Homeo}_+(\mathbb{R})$ be an injective homomorphism such that the action of $\Mod(\Sigma_g^1)$ on $\mathbb{R}$ is without global fixed points. Then, up to reversing orientation, $\rho$ is conjugate to a representation $\rho':\Mod(\Sigma_g^1) \rightarrow \mathrm{H}\widetilde{\mathrm{ome}}\mathrm{o}_+(S^1)$ such that $\rho'(T_d)(x) = x+1$ for all $x \in \mathbb{R}$ and $ c(h) = \tau^D(\rho'(h))$ for every $h \in \Mod(\Sigma_g^1)$.
\end{thm}
\begin{proof}
Suppose that $\rho:\Mod(\Sigma_g^1) \rightarrow \mathrm{Homeo}_+(\mathbb{R})$ is a homomorphism for which the corresponding action on $\RR$ has no global fixed point and such that $\rho(T_d)$ is not conjugate to shift by $\pm 1$.  Then $\rho(T_d)$ must have a fixed point, say $x_0$.  By ordering the cosets of the stabilizer $\mathrm{Stab}_{\rho}(x_0)$ in $\Mod(\Sigma_g^1)$ according to the orbit of $x_0$, and ordering $\mathrm{Stab}_{\rho}(x_0)$ however we please, we obtain a contradiction to Theorem \ref{thm:cofinal-boundary-twist} since $T_d \in \mathrm{Stab}_{\rho}(x_0)$, which is a convex subgroup in the resulting ordering.  Therefore, after fixing an appropriate orientation of $\RR$ we may choose $\rho':\Mod(\Sigma_g^1) \rightarrow \mathrm{H}\widetilde{\mathrm{ome}}\mathrm{o}_+(S^1)$ satisfying $\rho'(T_d)(x) = x+1$ for all $x \in \mathbb{R}$.  Now by Theorem \ref{main theorem}, for every $h \in \Mod(\Sigma_g^1)$ we have $\tau^D(\rho'(h)) = \tau^D(\rho_s(h)) = c(h)$.
\end{proof}

In particular, this means that the fractional Dehn twist coefficient of any element of $\Mod(\Sigma_g^1)$ can be computed directly from an arbitrary left ordering of $\Mod(\Sigma_g^1)$ (See also \cite{IK17}, where this result appears for the special case of the braid groups equipped with the Dehornoy ordering). In particular, the proof of Theorem \ref{fdtc from orders} in the introduction now follows immediately from Theorem \ref{conj to shift} and Proposition \ref{order to rep}.
%\begin{cor} 
%\label{fdtc from orders}
%Suppose $g \geq 3$ and fix a left ordering $<$ of $\Mod(\Sigma_g^1)$ for which $T_d >id$.  Then 
%\[ c(h) = \lim_{n \to \infty} \frac{[h^n]_<}{n}
%\]
%for all $h \in \Mod(\Sigma_g^1)$.
%\end{cor}
%\begin{proof}[Proof of Theorem \ref{fdtc from orders}]
%This is immediate from Theorem \ref{conj to shift} and Proposition \ref{order to rep}.
%\end{proof}

\subsection{Estimating fractional Dehn twist coefficients using left orderings}

In light of Proposition \ref{fdtc from orders}, every left ordering of $\Mod(\Sigma_g^1)$ gives rise to some easy techniques for estimating fractional Dehn twists.

\begin{prop}
\label{estimation}
Suppose that $g \geq 2$ and fix a left ordering $<$ of $\Mod(\Sigma_g^1)$ for which $T_d >id$.  If $T_d^k \leq h^m < T_d^{\ell}$ then $\frac{k}{m} \leq c(h) \leq \frac{\ell}{m}$.
\end{prop}
\begin{proof}
Because $T_d$ is central, the inequality $T_d^k \leq h^n < T_d^{\ell}$ implies that $T_d^{nk} \leq h^{nm} < T_d^{n\ell}$ for all $n>0$.  Therefore $nk \leq [h^{nm}]_{<} <n \ell$, and so 
\[  \frac{k}{m} \leq \lim_{n \to \infty} \frac{[h^{nm}]_<}{nm} \leq \frac{\ell}{m},
\]
but the central term is clearly equal to $c(h)$.
\end{proof}

Aside from yielding quick estimates of fractional Dehn twist coefficients, the fact that the previous proposition holds for every left ordering of $\Mod(\Sigma_g^1)$ allows for a new methods of computing fractional Dehn twist coefficients.

\begin{cor}
\label{estimation cor}
Suppose that $g \geq 2$, and let $<_1, <_2$ be left orderings of $\Mod(\Sigma_g^1)$ for which $T_d >_i id$ for $i=1,2$.  If $h^n \geq_1 T_d^k$ and there exists $f$ such that $fh^nf^{-1} \leq_2 T_d^k$, then $c(h) = \frac{k}{n}$.
\end{cor}
\begin{proof} This is a direct consequence of Propositions \ref{estimation} and \ref{conjugacy invariance}.
\end{proof}

In particular, this corollary implies that if there exists a left ordering $<$ of $\Mod(\Sigma_g^1)$ and $g, h \in \Mod(\Sigma_g^1)$ and $n \in \mathbb{Z}$, $n>0$ such that $[h^n]_< \neq [gh^ng^{-1}]_<$, then we can quickly determine the fractional Dehn twist coefficient of $h$.  

For if $T_d^k \leq h^n <T_d^{k+1}$, then it follows that $T_d^{k-1} \leq gh^ng^{-1} < T_d^{k+2}$.  So if $[h^n]_< \neq [gh^ng^{-1}]_<$ then it must be that either $T_d^{k-1} \leq gh^ng^{-1} < T_d^{k}$ or $T_d^{k+1} \leq gh^ng^{-1} < T_d^{k+2}$.  In the former case, $c(h) = \frac{k}{n}$, and in the latter, $c(h) = \frac{k+1}{n}$.

%While not at all obvious from the definition, the fractional Dehn twist coefficient is always an element of $\QQ$ {\bf cite needed}.  Therefore the previous theorem provides a surprising restriction on the actions of $\Mod(\Sigma_g^1)$ and $\Mod(\Sigma_{g,1})$ on $\RR$ and $\QQ$ respectively.
%
%\begin{cor}
%Suppose that $g \geq 3$.
%
%\begin{enumerate}
%\item If $\rho:\Mod(\Sigma_g^1) \rightarrow \mathrm{H}\widetilde{\mathrm{ome}}\mathrm{o}_+(S^1)$ is an injective homomorphism such that $\rho(T_d)(x) = x+1$ for all $x \in \mathbb{R}$ and such that the action of $\Mod(\Sigma_g^1)$ on $\mathbb{R}$ is without global fixed points, then $\tau^D_{\rho}(h) \in \QQ$ for all $h \in \Mod(\Sigma_g^1)$.
%\item  If $\rho:\Mod(\Sigma_{g,1}) \rightarrow \mathrm{Homeo}_+(S^1)$ is an injective homomorphism, then $\mathrm{rot}^D_{\rho}(h) \in \QQ$ for all $h \in \Mod(\Sigma_{g,1})$.
% \end{enumerate}
%\end{cor}

\section{Surfaces with many boundary components, low genus and marked points}
\label{limitations}

%\tg{I changed all references to Theorem 17 to Theorem 2, just to make it clearer we're talking about failures of our main theorems.}

In this brief section, we provide examples that show Theorem \ref{conj to shift intro version} and its left-orderability counterpart Theorem \ref{fdtc from orders} cannot hold for any surface $\Sigma_{g,n}^b$ with $b>1$, nor for surfaces $\Sigma_g^1$ when $g < 2$.   Whether or not our results hold for $\Sigma_{g,n}^1$ when $n>0$ and $g>1$ remains open. 

Our primary tool for doing so is the following observation.

\begin{prop}
\label{counterexample fixed points}
Let $b>0$, fix a boundary component $C$ of $\Sigma_{g,n}^b$, and recall $\rho_{s,C} : \Mod(\Sigma_{g,n}^b) \rightarrow \mathrm{H}\widetilde{\mathrm{ome}}\mathrm{o}_+(S^1)$ denotes the standard action constructed as in Section \ref{FDTC}.  If $\alpha$ is a simple closed curve in $\Sigma_{g,n}^b$ that is not isotopic into $C$, then $\mathrm{Fix}(\rho_{s,C}(T_{\alpha})) \neq \emptyset$.
\end{prop}
\begin{proof}
The homeomorphism $T_{\alpha}$ can be supported in a small annular neighbourhood $A$ of $\alpha$. Fix an infinite geodesic ray $\gamma$ in $ \Sigma_{g,n}^b$ beginning at $x_0 \in C$ and not entering $A$, for instance by taking $\gamma$ to wind around one side of the annulus $A$.  Then the lift $\tilde{\gamma}$ in the universal cover ends at a point in $\partial \widetilde{\Sigma_{g,n}^b}$ which is a fixed point of $\rho_{s,C}(T_{\alpha})$. 
\end{proof}

Now suppose that $b>1$ and choose distinct  boundary components $C, C' \subset \partial \Sigma_{g,n}^b$.  By Proposition \ref{counterexample fixed points},  we know that $c(T_{C'}, C)$ is zero, while $c(T_{C'},C')=1$.  Therefore Theorem \ref{conj to shift intro version} cannot hold for a surface with multiple boundary components.

We handle the cases of low-genus surfaces similarly.  Considering $\Mod(\Sigma_{0,n}^1)$ with $n>2$, choose $\alpha$ to be a simple closed curve encircling precisely two of the marked points.  Then $c(T_{\alpha}) = 0$ by Proposition \ref{counterexample fixed points}.  On the other hand, if $\phi: \Mod(\Sigma_{0,n}^1) \rightarrow \mathbb{Z}$ is the abelianisation map, then $\phi(T_{\alpha})$ is nonzero (it is a square of a generator of $\mathbb{Z})$, and so via the abelianisation we can construct an action of $\Mod(\Sigma_{0,n}^1)$ on $\mathbb{R}$ such that $T_{\alpha}$ has no fixed points.  

Similarly, considering $\Mod(\Sigma_{1,n}^1)$ where $n \geq 0$, we let $T_\alpha \in \Mod(\Sigma_{1,n}^1)$ denote the class of a Dehn twist along a nonseparating simple closed curve $\alpha$ in $\Sigma_{1,n}^1$.  Then Proposition \ref{counterexample fixed points} shows that $c(T_{\alpha}) = 0$.  On the other hand, we can choose $\alpha$ so that the abelianisation homomorphism provides a map $\Mod(\Sigma_{1,n}^1) \rightarrow \ZZ$ such that $T_\alpha \mapsto 1$ \cite[Section 5]{Korkmaz02}.  As in the previous paragraph, this results in an action of $\Mod(\Sigma_{1,n}^1)$ on $\mathbb{R}$ where $T_{\alpha}$ acts without fixed points.

We conclude that Theorem \ref{conj to shift intro version} does not hold for both $\Mod(\Sigma_{0,n}^1)$ with $n>2$ and $\Mod(\Sigma_{1,n}^1)$ where $n \geq 0$.

\end{document}